\documentclass{amsart}
\usepackage{amsmath,amsfonts,amssymb}
\usepackage[all]{xy}
\usepackage{graphicx}
\usepackage{yhmath}

\numberwithin{equation}{section}

\newtheorem{theorem}{Theorem}[section]
\newtheorem{definition}[theorem]{Definition}

\newtheorem{lemma}[theorem]{Lemma}
\newtheorem{remark}[theorem]{Remark}
\newtheorem{proposition}[theorem]{Proposition}
\newtheorem{example}[theorem]{Example}
\newtheorem{conjecture}[theorem]{Conjecture}

\renewenvironment{proof}[1][\proofname. ]{ { \noindent \it #1}}{\qed \\}

\newcommand{\bu}{\bold{u}}
\newcommand{\be}{\bold{e}}

\newcommand{\PP}{\mathbb{P}}

\newcommand{\RR}{\mathbb{R}}
\newcommand{\CC}{\mathbb{C}}
\newcommand{\ZZ}{\mathbb{Z}}

\newcommand{\LL}{\mathbb{L}}

\newcommand{\CO}{\mathcal{CO}}
\newcommand{\CA}{\mathcal{A}}
\newcommand{\CB}{\mathcal{B}}
\newcommand{\CF}{\mathcal{F}}
\newcommand{\CG}{\mathcal{G}}

\newcommand{\CM}{\mathcal M}
\newcommand{\CN}{\mathcal N}
\newcommand{\LM}{\mathcal{LM}}
\newcommand{\ks}{\mathfrak{ks}}
\newcommand{\PO}{\mathfrak{PO}}
\newcommand{\OL}[1]{\overline{#1}}
\newcommand{\UL}[1]{\underline{#1}}

\newcommand{\AI}{A_\infty}
\newcommand{\calC}{\mathcal{C}}
\newcommand{\calD}{\mathcal{D}}

\newcommand{\locmir}{\mathcal{L}\mathcal{M}}

\newcommand{\fb}{\mathfrak{b}}
\newcommand{\fq}{\mathfrak{q}}

\newcommand{\mfai}{MF_{\AI}}
\newcommand{\id}{{\mathrm{id}}}

\newcommand{\Hom}{{\rm Hom}}
\newcommand{\OCO}{\otimes \cdots \otimes}

\begin{document}
\markboth{Sangwook Lee}
{Ring isomorphisms of closed state spaces via homological mirror symmetry}

%\catchline{}{}{}{}{}

\title{RING ISOMORPHISMS OF CLOSED STATE SPACES VIA HOMOLOGICAL MIRROR SYMMETRY}
\author{SANGWOOK LEE}
\address{Korea Institute for Advanced Study\\ 85 Hoegi-ro Dongdaemun-gu, Seoul 02455, Korea\\ swlee@kias.re.kr}
\maketitle

\begin{abstract}
We investigate how ring isomorphisms between mirror closed states of compact toric Fano manifolds come from homological mirror symmetry. A natural closed-open relation on $B$-side is also discussed.

\end{abstract}

%\keywords{Homological mirror symmetry; closed string mirror symmetry; Hochschild cohomology}

%\ccode{Mathematics Subject Classification 2000: 53D40, 14J32}

%\tableofcontents
\section{Introduction}
\subsection{Mirror symmetry}
Mirror symmetry, which is one of central concepts in string theory, has also enabled mathematicians to predict a number of striking coincidences of structures which come from different origins. It was first recognized as a powerful tool in mathematics when it successfully predicted curve countings in Calabi-Yau manifolds(\cite{CdGP}) in terms of period integrals. If we consider a Fano manifold $X$ instead of Calabi-Yau manifolds, then a version of (closed string) mirror symmetry conjecture translates into the following ring isomorphism:
\begin{equation}\label{closedms}
 QH^*(X) \cong Jac(W)
 \end{equation}
where the holomorphic function $W$ is the {\em Landau-Ginzburg mirror} of $X$. We will call these algebras as closed state spaces(quantum cohomologies are $A$-models, while Jacobian rings are $B$-models). Obviously the isomorphism is surprising: it compares a symplectic geometry data(counting of rational curves) on one hand and a deformation theoretic invariant of a holomorphic function on the other hand. There have been several works concerning this statement, for example \cite{Ba,Gi}.

Later Kontsevich introduced the mirror conjecture for open strings in mathematical language, so-called the {\em homological mirror symmetry}. For Calabi-Yau mirror pairs $(X,\check{X})$, it compares two $\AI$-categories, one is the Fukaya category($A$-model) $ Fu(X)$ and the other is $D^b_\infty Coh(\check{X})$, the (dg-enhancement of) derived category of coherent sheaves($B$-model) on $\check{X}$. For a Fano manifold, if the $A$-model is its Fukaya category, then the $B$-model is a category of matrix factorizations $MF(W)$ of the Landau-Ginzburg mirror $W$.

We focus on compact toric manifolds. In this case, Fukaya-Oh-Ohta-Ono\cite{FOOOtoric3} constructed the mirror isomorphism (\ref{closedms}) in terms of disc countings. They show that every ambient cycle can be used as a ``bulk parameter", so that the obstruction $m_0^b$ deforms to a new multiple of the unit $m_0^{\fb,b}$ if $b$ is a weak bounding cochain. Furthermore any derivative of $m_0^{\fb,b}$ with respect to a direction in ambient cycles is still a multiple of the unit, and the coefficients are used to define the isomorphism. 

For toric Fano manifolds, we have homological mirror symmetry. In \cite{CHL2} they construct a functor from the Fukaya category of a toric Fano manifold to the category of matrix factorizations. The functor sends a full subcategory $\mathcal{C}$ generated by strongly balanced torus fibers to the whole category, because image objects are easily seen to be generators. On the other hand, in \cite{EL} Evans-Lekili proved a generation result for monotone Hamiltonian $G$-manifolds: in particular we deduce from their result that the subcategory $\mathcal{C}$ generates the Fukaya category, so we obtain the homologica mirror symmetry for Fano case. A more general theorem of the generations of Fukaya categories of compact toric manifolds has been announced by \cite{AFOOO}.

\subsection{Closed-open relations}
It has been widely believed that the homological mirror symmetry ``induces" the closed string mirror symmetry. In particular, if we concentrate on the ring structure of mirror closed state spaces, then we expect that they are isomorphic to {\em Hochschild cohomology algebras} of mirror categories. If a holomorphic function $W$ has isolated singularities, then we indeed have a ring isomorphism
\[\gamma: HH^*(MF(W)) \cong Jac(W)\]
which can be seen as a ``closed-open map" on the Landau-Ginzburg $B$-model. We will construct $\gamma$ explicitly and justify why it deserves the name.

For Fukaya categories, it is in general hard to prove that their Hochschild cohomologies are isomorphic to quantum cohomologies. By \cite{EL} again, for monotone Hamiltonian $G$-manifolds we have isomorphisms between them, by a corollary of the generation result.

\subsection{Result of this paper}
We study how $\ks$ can be realized as a ring isomorphism from homological mirror symmetry. More precisely, we prove the following.
\begin{theorem}
Let $X$ be a compact toric Fano manifold and $\PO$ be its LG superpotential. Then there is a $Fu(X)$-$\OL{\mfai}(\PO)$-bimodule $\CM$ such that the following diagram commutes:
\[\xymatrix{
QH^*(X) \ar[r]^-{\mathcal{CO}} \ar[dd]^{\ks} & HH^*(Fu(X))\ar[rd]^\cong_{L_\CM^1} & \\
& & Hom†(\CM,\CM). \\
\bigoplus_{\eta \in Crit(\PO)} Jac(\PO;\eta) \ar[r]^-{\gamma} & HH^*(\OL{\mfai}(\PO)) \ar[ur]^\cong_{R_\CM^1} & }
\]
\end{theorem}

%We remark that the proof of $\ks$ being a ring isomorphism is elegantly done in \cite{FOOOtoric3} by examining degenerations of holomorphic discs with interior markings. Our main point is that we can interprete it as a ``shadow" of homological mirror symmetry.
We will construct $\CM$ explicitly from the localized mirror functor between $Fu(X)$ and $\overline{\mfai}(\PO)$. The maps $L^1_\CM$ and $R^1_\CM$ from Hochschild cohomologies to the endomorphism space of $\CM$ are described in \cite{She15} and will be reviewed. The main content of the proof begins with an explicit description of the map $\gamma$ in terms of the bar resolution. Then the proof is a usual check of $\AI$-relations together with their differentiations. 

We organize the paper as follows: in Section \ref{sec:algebra} we recall algebraic preliminaries concerning $\AI$-categories. Then we recollect relevant definitions and facts of Lagrangian Floer theory on compact toric manifolds in Section \ref{sec:lagfloer}. Section \ref{sec:mfctg} deals with categories of matrix factorizations. In Section \ref{sec:mirror} we study mirror symmetries, both closed and open string ones. Finally, we state and prove our main theorem in Section \ref{sec:mainthm}.

\subsection*{Acknowledgements}
The author is grateful to Cheol-hyun Cho for his encouragement and interest in this work. He also thanks Bumsig Kim who invited him to KIAS where he began to work on this subject. The author was supported by NRF-2007-0093859.

%The organization of this paper is as follows. We start with algebraic preliminaries in Section 2 which are the most relevant in this work. We will encounter several definition concerning $\AI$-categories and bimodules. Then we recall Hochschild cohomologies and structures on them. In Section 3 we review homological mirror symmetry and relevant categories. In Section 4 we study Fukaya-Oh-Ohta-Ono's Kodaira-Spencer map(with emphasis on its geometric nature). In the final section we prove our main theorem.

\section{$\AI$-categories and $\AI$-bimodules}\label{sec:algebra}

\subsection{Basic definitions}
We refer readers to \cite{Gan,Seidelbook,She15} for more detailed explanations and proofs of the results discussed in this section.
\begin{definition}

The {\em Novikov field} is $\displaystyle\Lambda:=\Big\{\sum_{i \geq 0} a_i T^{\lambda_i} \mid a_i \in \CC, \lambda_i \in \RR, \lambda_i \to \infty \;\mathrm{as} \; i \to \infty \Big\}.$

A filtration $F^\bullet\Lambda$ of $\Lambda$ is given by \[F^\lambda \Lambda:=\Big\{\sum_{i\geq 0}a_i T^{\lambda_i} \mid \lambda_i \geq \lambda {\rm \;for\; all\; }i \Big\} \subset \Lambda.\]

We write \[F^+\Lambda:=\Big\{\sum_{i\geq 0}a_i T^{\lambda_i} \mid \lambda_i > 0 {\rm \;for\; all\; }i \Big\}.\]

The {\em Novikov ring} $\Lambda_0$ is $F^0\Lambda$.
%For later use we also define a graded ring \[\Lambda^{gr}:=\Lambda[e,e^{-1}]=\Big\{\sum_{i \geq 0} a_i T^{\lambda_i}e^{\mu_i} \mid a_i \in \CC, \lambda_i \in \RR, \lambda_i \to \infty \;\mathrm{as} \; i \to \infty, \mu_i \in \ZZ, {\rm deg}(e)=2 \Big\}.\]
\end{definition}

\begin{definition}
A {\em filtered $\AI$-category} $\calC$ over $\Lambda$ consists of a class of objects $Ob(\calC)$ and the set of morphisms $hom_{\calC}(A,B)$ for each pair of objects $A$ and $B$(we write $\calC(A,B)$ instead of $hom_\calC(A,B)$ for simplicity), with the following conditions:

\begin{enumerate}
 \item $\calC(A,B)$ is a filtered $\ZZ/2$-graded $\Lambda$-vector space for any $A,B\in Ob(\calC)$,
 \item for $k \geq 0$ there are multilinear maps of degree $1$
  \[m_k: \calC(A_0,A_1)[1] \otimes \calC(A_1,A_2)[1] \otimes \cdots \otimes \calC(A_{k-1},A_k)[1] \to \calC(A_0,A_k)[1]\]
  such that they preserve the filtration and satisfy the {\em $\AI$-relation}
  \begin{equation}\label{aieq}\sum_{k_1+k_2=n+1}\sum_{i=1}^{k_1}(-1)^\epsilon m_{k_1}(x_1,...,x_{i-1},m_{k_2}(x_i,...,x_{i+k_2-1}),x_{i+k_2},...,x_n)=0\end{equation} where $\epsilon=\sum_{j=1}^{i-1}(|x_j|+1).$ 
\end{enumerate}
The meaning of $m_0$ is that for each object $A$ we have $m_0^A \in hom^1(A,A)[1]=hom^2(A,A).$
%If $m_0 \neq 0$, $\calC$ is called a {\em curved} $\AI$-category. Otherwise, $\calC$ is called {\em strict.}
%If there is only one object, then $\calC$ is called an {\em $\AI$-algebra.} If only $m_1$ and $m_2$ are nonzero, then $\calC$ is called a {\em dg category.}
\end{definition}
For convenience, we write $|x_j|':=|x_j|+1$.

We can also define $\AI$-operations by coderivations. Let 
\[B\calC:=\bigoplus\limits_{X_0,\cdots,X_k \in Ob(\calC)} \calC(X_0,X_1)[1]\otimes\cdots\otimes \calC(X_{k-1},X_k)[1].\]
Then define 
\[ \hat{m}: B\calC \to B\calC,\]
\[ \hat{m}(x_0,\cdots,x_{k-1})=\sum_{l,j} (-1)^{\sum_{i=0}^l |x_i|'} x_0 \otimes \cdots \otimes x_l \otimes m_j(x_{l+1},\cdots,x_{l+j}),x_{l+j+1}\otimes \cdots\otimes x_{k-1}.\]
Now (\ref{aieq}) is equivalent to $\hat{m} \circ \hat{m}=0.$

\begin{definition}
Let $\calC$ and $\calC'$ be $\AI$-categories. An {\em $\AI$-functor} between $\calC$ and $\calC'$ is a collection $\CF=\{\CF_i\}_{i \geq 0}$ consisting of

\begin{itemize}
 \item $\CF_0: Ob(\calC) \to Ob(\calC'),$
 \item $\CF_k: \calC(A_0,A_1)[1] \otimes \cdots \otimes \calC(A_{k-1},A_k)[1] \to \calC'(\CF_0(A_0),\CF_0(A_k))[1]$, mulitlinear maps of degree $0$
\end{itemize}
which are subject to the following condition:
\begin{eqnarray*}
& \sum\limits_{i,j} (-1)^{|x_1|'+\cdots+|x_{i-1}|'} \CF_{i-j+k}(x_1,...,x_{i-1},m^\calC_{j-i+1}(x_i,...,x_j),x_{j+1},...,x_k)\\
=& \sum\limits_l m^{\calC'}_{l+1}(\CF_{i_1-1}(x_1,...,x_{i_1}),\CF_{i_2-i_1}(x_{i_1+1},...,x_{i_2}),...,\CF_{k-i_l}(x_{i_l+1},...,x_{k})).
\end{eqnarray*}

\end{definition}
Again, we can define an $\AI$-functor $\CF$ as an equation 
\[ \CF\circ \hat{m}=m\circ \hat{\CF}\]
where 
\[\hat{\CF}:B\calC \to B\calC,\]
\[ \hat{\CF}(x_1,\cdots,x_k)=\sum_{i_1,\cdots,i_l} \CF_{i_1}(x_1,\cdots,x_{i_1}) \otimes \cdots \otimes \CF_{i_l}(x_{k-i_l+1},\cdots,x_k).\]

For $x \in hom(A,B)$, we have
\begin{equation}\label{eq:m1}m_1^2(x)+m_2(m_0^A,x)+(-1)^{|x|+1}m_2(x,m_0^B)=0.\end{equation}
Hence if $m_0=0$, then $m_1^2=0$, but if $m_0$ is nonzero then $m_1$ does not have to be a differential. 
%\begin{remark}Even if $m_0=0$, from the $\AI$-relation the composition map $m_2$ might not satisfy associativity. So in general an $\AI$-category is {\em not} a category in the usual sense. \end{remark}

To achieve $m_1^2 =0$ so that we can think of $m_1$-homologies, we need to deform the original $\AI$-category. We recall more definitions.

\begin{definition}
For an object $A$ in an $\AI$-category $\calC$, $e_A \in \calC(A,A)$ is called a {\em unit} if it satisfies
 \begin{enumerate}
  \item $m_2(e_A,x)=x,\; m_2(y,e_A)=(-1)^{|y|}y$ for any $x \in \calC(A,B)$, $y \in \calC(B,A)$,
  \item $m_{k+1}(x_1,...,e_A,...,x_k)=0$ for any $k \neq 1$.
 \end{enumerate}
\end{definition}

\begin{definition}\label{def:wbc}
An element $b \in F^+\calC^1(A,A)$ is called a {\em weak bounding cochain} of $A$ if it is a solution of the weak Maurer-Cartan equation
\begin{equation}\label{eq:weakMC}
m(e^b):=m_0^A+m_1(b)+m_2(b,b)+ \cdots = \PO(A,b)\cdot e_A\end{equation}
for some $PO(A,b) \in \Lambda.$ If such a solution exists, then $A$ is called {\em weakly unobstructed}. 
 If there exists a solution $b$ such that $\PO(A,b)=0$, then $b$ is called a {\em bounding cochain} and $A$ is called {\em unobstructed}. $\PO(A,b)$ is called the {\em Landau-Ginzburg superpotential} of $b$. \end{definition}
We denote $\mathcal{M}_{weak}(A)$ be the set of weak bounding cochains of $A$. Then 
$\PO(A,\cdot)$ is a function on $\mathcal{M}_{weak}(A)$. We also define 
\[\CM_{weak}^\lambda(A):=\{b\in \CM_{weak}(A) \mid \PO(A,b)=\lambda\}.\]

Given an $\AI$-category $\calC$, under the assumption $\CM^\lambda_{weak}(A)$ is nonempty for some objects, we define a new $\AI$-category $\calC_\lambda$ as
\[Ob(\calC_\lambda)=\bigcup_{A \in Ob(\calC)}\{A\} \times \CM_{weak}^\lambda(A),\]
\[hom_{\calC_\lambda}((A_1,b_1),(A_2,b_2))=hom_{\calC}(A_1,A_2)\]
with the following $\AI$-structure maps
\[m_k^{b_0,...,b_k}:\calC_\lambda((A_0,b_0),(A_1,b_1)) \otimes \cdots \otimes {\calC_\lambda}((A_{k-1},b_{k-1}),((A_k,b_k)) \to {\calC_\lambda}((A_0,b_0),(A_k,b_k)),\]
\[m_k^{b_0,...,b_k}(x_1,...,x_k):=\sum_{l_0,...,l_k}m_{k+l_0+\cdots+l_k}(b_0^{l_0},x_1,b_1^{l_1},...,b_{k-1}^{l_{k-1}},x_k,b_k^{l_k})\]
where $x_i \in {\calC_\lambda}((A_i,b_i),(A_{i+1},b_{i+1})).$ $\AI$-relation is induced by the weak Maurer-Cartan equation (\ref{eq:weakMC}).

\begin{theorem}
Let $(A_0,b_0),(A_1,b_1) \in Ob(\calC_\lambda)$. Then $(m_1^{b_0,b_1})^2=0.$
\end{theorem}

\begin{proof}
Let $x \in {\calC_\lambda}((A_0,b_0),(A_1,b_1)).$ Then the $\AI$-equation is \[(m_1^{b_0,b_1})^2+m_2(m(e^{b_0}),x)+(-1)^{|x|+1}m_2(x,m(e^{b_1}))=0.\] By $m(e^{b_0})=\lambda\cdot e_{A_0},$ $m(e^{b_1})=\lambda \cdot e_{A_1}$ and by definition of units, \[m_2(m(e^{b_0}),x)+(-1)^{|x|+1}m_2(x,m(e^{b_1}))=0,\] so $(m_1^{b_0,b_1})^2=0.$
\end{proof} 

Hence, if there exist weak Maurer-Cartan solutions, we get a new $\AI$-category with $m_1^2=0$ by restricting to objects sharing certain value of a Landau-Ginzburg(LG for short) superpotential.

\begin{definition}
Let $\calC$ and $\calD$ be $\AI$-categories. A {\em $\calC$-$\calD$-bimodule} $\CM$ is the following data:
\begin{itemize}
\item For $V \in Ob(\calC)$ and $V' \in Ob(\calD)$, $\CM(V,V')$ is a graded $k$-vector space,
\item degree $1-r-s$ multilinear maps
\begin{align*} 
\mu^{r|1|s}: &\; \calC(V_0,V_1) \OCO \calC(V_{r-1},V_r)\otimes \CM(V_r,W_0)\otimes \calD(W_0,W_1)\OCO \calD(W_{s-1},W_s) \\
&\to \CM(V_0,W_s)\end{align*}
for any $V_i \in Ob(\calC)$ and $W_j \in Ob(\calD)$ satisfying
\begin{align*}
& \sum (-1)^{\epsilon}\mu^{\tiny{i+1|1|s-j-1}}(v_0,\cdots,v_i,\mu^{r-i-1|1|j+1}(v_{i+1},\cdots,v_{r-1},\UL{m},w_0,\cdots,w_j),w_{j+1},\cdots,w_{s-1}) \\
&+ \sum (-1)^{\epsilon}\mu^{i+r-j+1|1|s}(v_0,\cdots,v_i,m_{j-i}^\calC(v_{i+1},\cdots,v_j),v_{j+1},\cdots,v_{r-1},\UL{m},w_0,\cdots,w_{s-1}) \\
&+ \sum (-1)^{\epsilon'} \mu^{r|1|i+s-l+1}(v_0,\cdots,v_{r-1},\UL{m},w_0,\cdots,w_i,m_{l-i}^\calD(w_{i+1},\cdots,w_l),w_{l+1},\cdots,w_{s-1}) \\
&=0
\end{align*}
where 
\[
\epsilon=|v_0|'+\cdots+|v_i|', \;\;
\epsilon'=|v_0|'+\cdots+|v_{r-1}|'+|m|+|w_0|'+\cdots+|w_i|'.
\]
\end{itemize}
\end{definition}

\begin{definition}
A {\em pre-morphism of $\calC$-$\calD$-bimodules $\CF: \CM \to \CM'$ of degree $k$} is a collection of multilinear maps 
\begin{align*}
 \CF^{r|1|s}:&\;\calC(V_0,V_1)\OCO \calC(V_{r-1},V_r) \otimes \CM(V_r,W_0) \otimes \calD(W_0,W_1) \OCO \calD(W_{s-1},W_s) \\
 &\to \CM'(V_0,W_s)\end{align*}
of degree $k-r-s$, and the {\em composition} $\CF' \circ \CF$ is defined by
\begin{align*}
&(\CF' \circ \CF)(v_1,\cdots,v_r,\UL{m},w_1,\cdots,w_s) \\
:= & \sum (-1)^{|\CF|(|v_1|'+\cdots+|v_i|')}\CF'^{i|1|s-j}(v_1,\cdots,v_i,\CF^{r-i|1|j}(v_{i+1},\cdots,\UL{m},w_1,\cdots,w_j),w_{j+1},\cdots,w_s).
\end{align*}

The {\em differential} $\delta$ on premorphisms is defined by
\begin{align*}
&(\delta\CF)(v_1,\cdots,v_r,\UL{m},w_1,\cdots,w_s) \\
:= &\sum (-1)^{\epsilon_1} \mu^{i|1|s-j}_{\CM'}(v_1,\cdots,v_i,\CF^{r-i|1|j}(v_{i+1},\cdots,v_r,\UL{m},w_1,\cdots,w_j),w_{j+1},\cdots,w_s) \\
& -\sum (-1)^{\epsilon_2} \CF^{i|1|s-j}(v_1,\cdots,v_i,\mu^{r-i|1|s-j}(v_{i+1},\cdots,v_r,\UL{m},w_1,\cdots,w_j),w_{j+1},\cdots,w_s) \\
& - \sum \CF^{*|1|s}(\hat{m}^\calC(v_1,\cdots,v_r),\UL{m},w_1,\cdots,w_s) \\
& -\sum (-1)^{\epsilon_3} \CF^{r|1|*}(v_1,\cdots,v_r,\UL{m},\hat{m}^\calD(w_1,\cdots,w_s))
\end{align*}
where 
\[ \epsilon_1=|\CF|(|v_1|'+\cdots+|v_i|'),\; \epsilon_2=|v_1|'+\cdots+|v_i|',\; \epsilon_3=|v_1|'+\cdots+|v_r|'+|m|.\]
\end{definition}
More succintly, 
\[ \delta\CF=\mu_{\CM'}\circ \hat{\CF}-(-1)^{|\CF|}\CF \circ \hat{\mu}_\CM.\]
\begin{remark}
The definition of the degree of a pre-morphism of bimodules is motivated by the fact that the degree $k$ pre-morphism is indeed given by multilinear maps
\[ \calC[1]^{\otimes r} \otimes \CM \otimes \calD[1]^{\otimes s} \to \CM\]
of degree $k$. Once we accept such a perspective on degrees of maps, all signs obey only Koszul conventions and so we sometimes just write $(-1)^{{\rm Koszul}}$ for signs when it is sufficient. %Also note that in the 3rd and the 4th lines of the definition of $\delta \CF$ we did not specify signs occured by hat maps.
\end{remark}
The readers can easily check that $\calC$-$\calD$-bimodules together with premorphisms form a dg category. If $\CF: \CM \to \mathcal{N}$ is a premorphism such that $\delta \CF=0$ and its cohomology level map $[\CF^{0|1|0}]$ is an isomorphism, then $\CF$ is called a {\em quasi-isomorphism} and write $\CM \cong \mathcal{N}$.

\begin{example}
For an $\AI$-category $\calC$, the {\em diagonal bimodule} $\calC_\Delta$ is a $\calC$-$\calC$-bimodule defined by
\[ \calC_\Delta(X,Y):=\calC(X,Y)\]
with module structure maps 
\[ \mu^{r|1|s}:=m^\calC_{r+s+1}.\]
\end{example}

We recall some operations on bimodules.
\begin{definition}[Base change]
Let $\CF: \calC_1 \to \calC_2$ and $\CG: \calD_1 \to \calD_2$ be $\AI$-functors. Suppose that $\CM$ is a $\calC_2$-$\calD_2$-bimodule. Then a $\calC_1$-$\calD_1$-bimodule $(\CF\otimes \CG)^*\CM$ is defined on objects by
\[(\CF\otimes \CG)^*\CM(X,Y):=\CM(\CF(X),\CG(Y))\] for $X \in Ob(\calC_1)$, $Y\in Ob(\calD_1)$, and the structure maps are given by
\begin{align*}
&\mu^{r|1|s}_{(\CF\otimes \CG)^*\CM}(v_1,\cdots,v_r,\UL{m},w_1,\cdots,w_s) \\
:=&\sum_{k,l} \mu^{k|1|l}_\CM\big(\CF_{i_1}(v_1,\cdots),\cdots,\CF_{i_k}(\cdots,v_r),\UL{m},\CG_{j_1}(w_1,\cdots),\cdots,\CG_{j_l}(\cdots,w_s)\big).
\end{align*}
\end{definition}

%We also define the tensor product of bimodules just for completeness(for the definition of Morita equivalence), but we do not use the definition in the sequel.
\begin{definition}[Tensor product]
Let $\CM$ be a $\calC$-$\calD$-bimodule and $\mathcal{N}$ be a $\calD$-$\mathcal{E}$-bimodule for $\AI$-categories $\calC$, $\calD$ and $\mathcal{E}$. Then $\CM\otimes_\calD \mathcal{N}$ is a $\calC$-$\mathcal{E}$-bimodule such that
\[ (\CM \otimes_\calD \mathcal{N})(C,E)
:=\bigoplus_{D_1,...,D_k \in Ob(\calD)}\CM(C,D_1)\otimes \calD(D_1,D_2)\otimes \cdots \otimes
\calD(D_{k-1},D_k)\otimes \mathcal{N}(D_k,E)\]
for $C\in Ob(\calC)$ and $E \in Ob(\mathcal{E})$, and the structure maps are given as follows:
\begin{align*}
 &\mu^{0|1|0}_{\CM \otimes_\calD \mathcal{N}}(\UL{m}\otimes d_1 \OCO d_k \otimes \UL{n}) \\
 := & \sum \mu_\CM^{0|1|i}(\UL{m}, d_1,\cdots, d_i)\otimes d_{i+1} \OCO d_k \otimes \UL{n} \\
 & +\sum (-1)^{\rm Koszul} \UL{m} \otimes d_1 \OCO d_i \otimes m_{j-i}^\calD(d_{i+1},\cdots,d_j) \otimes d_{j+1} \OCO d_k \otimes \UL{n} \\
 & + \sum (-1)^{\rm Koszul} \UL{m} \otimes d_1 \OCO d_i \otimes \mu_{\mathcal{N}}^{k-i|1|0}(d_{i+1},\cdots, d_k,\UL{n}), \\ \\
&\mu_{\CM \otimes_\calD \mathcal{N}}^{r|1|0}(c_1,\cdots,c_r,\UL{m}\otimes d_1 \OCO d_k \otimes \UL{n}) \\
:= & \sum \mu_\CM^{r|1|p}(c_1,\cdots,c_r,\UL{m},d_1,\cdots,d_p)\otimes d_{p+1}\otimes \cdots \otimes d_k \otimes \UL{n}, \\ \\
 & \mu_{\CM \otimes_\calD \mathcal{N}}^{0|1|s}(\UL{m} \otimes d_1 \OCO d_k \otimes \UL{n},e_1,\cdots,e_s) \\
:= & \sum (-1)^{\rm Koszul} \UL{m} \otimes d_1 \otimes \cdots \otimes d_p \otimes \mu_\mathcal{N}^{k-p|1|s}(d_{p+1},\cdots,d_k,\UL{n},e_1,\cdots,e_s)
\end{align*}
and $\mu_{\CM \otimes_\calD \mathcal{N}}^{r|1|s}=0$ if $r$ and $s$ are both nonzero.
\end{definition}

\begin{definition}
Two $\AI$-categories $\calC$ and $\calD$ are {\em Morita equivalent} if there is a $\calC$-$\calD$-bimodule $\CM$ and a $\calD$-$\calC$-bimodule $\mathcal{N}$ such that
\[ \CM \otimes_\calD \mathcal{N} \cong \calC_\Delta, \;\; 
\mathcal{N} \otimes_\calC \CM \cong \calD_\Delta.\]
In this case, we call $\CM$ and $\CN$ be {\em Morita bimodules.}
\end{definition}

\subsection{Hochschild cohomology}
\begin{definition}
Let $\CM$ be an $\AI$-bimodule over $\calC$. We define the {\em Hochschild cochain complex} of $\CM$ 
\[ CH^*(\calC,\CM):=\prod_{X_0,...,X_k \in Ob(\calC)} hom^\bullet\big(\calC(X_0,X_1)[1] \OCO \calC(X_{k-1},X_k)[1],\CM(X_0,X_k)\big)[-1]
\]
with differential $b^*$ defined by
\begin{align*}
& b^*\phi (x_0,\cdots,x_{k-1}) \\
:=&  \sum \phi(\hat{m}(x_0,\cdots,x_{k-1})) \\
 &+\sum (-1)^{|\phi|'(|x_0|'+\cdots+|x_i|')} \mu_\CM^{i+1|1|k-l-1}(x_0,\cdots,x_i,\phi(x_{i+1},\cdots,x_l),x_{l+1},\cdots,x_{k-1}).
\end{align*}
Its cohomology of $b^*$ is called the {\em Hochschild cohomology of $\calC$ with coefficient in $\CM$.} If $\CM=\calC_\Delta$, then we write $CH^*(\calC):=CH^*(\calC,\calC_\Delta).$

\end{definition}
%\begin{example}
%The $\AI$-structure maps $\{ m_k\}$ of $\calC$ form a Hochschild cocycle in $CH^*(\calC)$.
%\end{example}

\begin{proposition}[\cite{Get}]
$CH^*(\calC)$ is an $\AI$-algebra with $\AI$-operations given by
\begin{align*} 
&M^k(\phi_1,\cdots,\phi_k)(x_1,\cdots,x_n)\\
:= &\sum (-1)^{\epsilon} m_*(\vec{x}_{i_1},\phi_1(\vec{x}_{j_1}),\vec{x}_{i_2},\phi_2(\vec{x}_{j_2}),\cdots,\vec{x}_{i_k},\phi_k(\vec{x}_{j_k}),\vec{x}_{i_{k+1}}).
\end{align*}
with $\vec{x}_{i_1}\otimes \vec{x}_{j_1}\otimes \cdots \otimes \vec{x}_{i_{k+1}}=x_1\OCO x_n,$ $M^0=0$, $M^1=b^*,$ and
\[ \epsilon=\sum_{l=1}^k |\phi_l|'(|\vec{x}_{i_1}|'+|\vec{x}_{j_1}|'+\cdots+|\vec{x}_{j_{l-1}}|'+|\vec{x}_{i_l}|').\]
\end{proposition}
In particular, the binary product $M^2$ induces the {\em Yoneda product} $\cup$ on the cohomology $HH^*(\calC)$ by
\[ \phi \cup \psi:=(-1)^{|\phi|} M^2(\phi,\psi).\]
Then the Yoneda product is associative. 

\subsection{Morita invariance of Hochschild cohomology}
Let $\CA$ and $\CB$ be Morita equivalent with Morita bimodules $\CM$ and $\CN$ which are over $\CA$-$\CB$ and $\CB$-$\CA$ respectively. Recall the following:

\begin{lemma}[\cite{She15}]\label{lmrm}
The following are $\AI$ quasi-isomorphisms
\[ L_\CM: CH^*(\CA) \to hom^*_{\CA-\CB}(\CM,\CM),\]
\[ R_\CM: CH^*(\CB)^{op} \to hom^*_{\CA-\CB}(\CM,\CM)\]
which are defined as follows:
\begin{align*} 
&L_\CM^p(\phi_1,\cdots,\phi_p)(a_1,\cdots,a_r,\UL{m},b_1,\cdots,b_s) \\
:=& \sum (-1)^{\rm Koszul} \mu_\CM^{*|1|s}(\vec{a}_1,\phi_1(\vec{a}_2),\vec{a}_3,\cdots,\phi_p(\vec{a}_{2p}),\vec{a}_{2p+1},\UL{m},b_1,\cdots,b_s),
\end{align*}
\begin{align*}
&R_\CM^p(\phi_1,\cdots,\phi_p)(a_1,\cdots,a_r,\UL{m},b_1,\cdots,b_s) \\
:=& \sum (-1)^{\rm Koszul} \mu_\CM^{r|1|*}(a_1,\cdots,a_r,\UL{m},\vec{b}_1,\phi_p(\vec{b}_2),\vec{b}_3,\cdots,\phi_1(\vec{b}_{2p}),\vec{b}_{2p+1}).
\end{align*}
In particular, $L_\CM^1$ and $R_\CM^1$ induce ring isomorphisms on cohomology level.
\end{lemma}

Later we will examine the maps $L_\CM^1$ and $R_\CM^1$ carefully for the Morita bimodule between mirror pairs $Fu(X)$ and $\overline{\mfai}(W)$.

\section{Lagrangian Floer theory on compact toric manifolds}\label{sec:lagfloer}

\subsection{Review on Fukaya categories}
Let $(X,\omega)$ be a symplectic manifold and $L_0,\cdots,L_{k}$ be its Lagrangian submanifolds. Pick a compatible almost complex structure $J_0$ on $(X,\omega)$. For a while we only assume an extreme case of intersecting patterns of those Lagrangians: there are only transverse intersections among $L_0,\cdots,L_{k}$. But we will also deal with the other extreme case $L_0=\cdots=L_{k}$ later.

Let $CF(L_i,L_{i+1}):= \bigoplus\limits_{p \in L_i \cap L_{i+1}} \Lambda \cdot p$ be a $\ZZ/2$-graded vector space over $\Lambda$. The degree is defined by the Maslov index of each intersection point. Let $\beta \in \pi_2(X, L_0 \cup \cdots \cup L_k)$. Define a moduli space 
\begin{align*}
&\widehat{\CM}(p_0,\cdots,p_{k-1};q;\beta;J_0) \\
:=& \{u:(D^2,z_0,\cdots,z_{k-1},z_k) \stackrel{J_0{\rm -hol}}{\longrightarrow} (X,p_0,\cdots,p_{k-1},q) \mid \\
&\;\; u(\wideparen{z_i z_{i+1}}) \subset L_{i+1}, u(\wideparen{z z_0})\subset L_0, [u]=\beta \}.
\end{align*}

and
$\CM(p_0,\cdots,p_{k-1};q;\beta;J_0)$ be its stable map compactification. In general such moduli spaces may not be transversal, so we apply Kuranishi perturbations and achieve moduli spaces of correct dimensions. We denote such `correct' moduli space by $\CM(p_0,\cdots,p_{k-1};q;\beta)$, omitting $J_0$. Suppose that we can take suitable Kuranishi perturbations so that boundary components of the moduli spaces are expressed as fiber products of moduli spaces of the same kind with respect to evaluation maps at boundary marked points(these technical assumptions can be indeed achieved for compact toric manifolds, see \cite{FOOOtoric1,FOOOtoric2,FOOOtoric3}).

Now, define
\[ m_k(p_0,\cdots,p_{k-1}):= \sum_{\stackrel{\beta\in \pi_2(X,L_0\cup \cdots \cup L_k)}{q\in L_k \cap L_0}}\#\CM(p_0,\cdots,p_{k-1};q;\beta)\cdot T^{\omega(\beta)}\cdot q.\]
We only count the moduli space when the correct virtual dimension is zero. The above technical assumptions enables us to say that these operations satisfy a filtered $\AI$-relation (\ref{aieq}).

\begin{definition}
The {\em Fukaya category} $Fu(X)$ is an $\AI$-category whose objects are Lagrangian submanifolds and morphism spaces are $CF(L,L')$. The $\AI$-structure is given by above operations $\{m_k\}_{k\geq 0}$. The {\em derived Fukaya category} $D^\pi Fu(X)$ is the homotopy category of the triangulated envelope of $Fu(X)$ followed by Karoubi completion.
\end{definition}
In general, there might be nonzero $m_0$, which comes from holomorphic discs whose boundaries are on Lagrangian submanifolds. More precisely, for $\beta \in \pi_2(X,L)$, let $\CM_1^{main}(\beta)$ be the (Kuranishi perturbed) moduli space of holomorphic discs of homotopy classes $\beta$ with a boundary marked point. Define $m_{0,\beta}^L:=PD((ev_0)_* \CM_1(\beta))$ which is a cocycle in L, and let $m_0^L:=\sum\limits_{\beta \in \pi_2(X,L)}m_{0,\beta}^L.$ Higher $\AI$-{\em algebra} operations on a single Lagrangian are defined as follows. Let $\CM_{k+1}^{main}(\beta)$ be the moduli space of holomorphic discs of class $\beta$ with $k+1$ boundary marked points which respect the counterclockwise cyclic order. Let 
\[ev_0^\beta,\cdots,ev_k^\beta: \CM_{k+1}^{main}(\beta) \to L\] be evaluation maps at marked points $(z_0,\cdots,z_k)$ of each disc. Let $p_1,\cdots,p_{k}$ be cycles of $L$. Then 
\[m_{k,\beta}^L: H^*(L;\Lambda)^{\otimes k} \to H^*(L;\Lambda)\] is defined by
\[ (PD(p_1),\cdots PD(p_k)) \mapsto PD[(ev_0^\beta)_!(ev_1^\beta \times \cdots \times ev_k^\beta)^*(p_1 \times \cdots \times p_k)]. \]
Define \[ m_k^L:=\sum_\beta T^{\omega(\beta)}\cdot m_{k,\beta}^L.\]
We refer to \cite{FOOO,FOOOtoric1} for the well-definedness of above operations on the de Rham model of each single Lagrangian submanifold. If we take Kuranishi perturbations successfully, then the fundamental class $[L]$ plays the role of the $\AI$-unit, so the following definition makes sense(see Definition \ref{def:wbc}).

\begin{definition}
We say $L$ is {\em weakly unobstructed} if there is a chain $b\in C^*(L;\Lambda)$ such that for some constant $\PO(L,b)\in \Lambda$,
\[(m_0^L)^b:=m_0^L+m_1^L(b)+m_2^L(b,b)+\cdots=c\cdot [L].\]
The constant $\PO(L,b)$ is called the {\em superpotential} of $(L,b)$.
\end{definition}

\subsection{Floer theory on compact toric manifolds}
Let $X$ be a compact toric manifold and $\pi:X \to P$ be the moment map. For $\bu \in Int(P)$, let $L(\bu):=\pi^{-1}(\bu)$, i.e. a Lagrangian torus fiber over $\bu$. Recall that $P$ can be expressed as
\[ P=\{ \bu \in \RR^n \mid l_j(\bu)\geq 0, j=1,\cdots,m\}\]
for certain affine functionals $l_j: \RR^n \to \RR.$ Let $\vec{v}_j$ be the primitive normal vector to the facet $\partial_i P=\{ \bu\in \RR^n \mid l_j(\bu)=0\}.$ Let
\[ \vec{v}_j=(v_{j,1},\cdots,v_{j,n}).\]

Based on the classification of holomorphic discs bounded by torus fibers \cite{CO} and the transversality result \cite{FOOOtoric1}, we have the following substantial result.
\begin{theorem}[\cite{FOOOtoric1}]
\begin{itemize}
\item Every torus fiber $L(\bu)$ of a compact toric manifold $X$ is weakly unobstructed and $H^1(L(\bu);\Lambda_0)$ is contained in the weak Maurer-Cartan solution space.
\item Let $\{ \be_i\}_{i=1,\cdots,n}$ be the basis of $L(\bu)$, consisting of elements representing $dt_i$. Let $\{ x_i\}_{i=1,\cdots,n}$ be the dual basis and $b=\sum_i x_i \be_i$. Then the potential $\PO(L(\bu),b)$ has the expression
\[ \PO(L(\bu),b)=z_1+\cdots+z_m+\sum_{k=1}^N P_k(z_1,\cdots,z_m)\]
where $z_i:=T^{l_j(\bu)} (y_1^\bu)^{v_{j,1}}\cdots (y_n^\bu)^{v_{j,n}}$ and $y_1^\bu:=e^{x_i}$.
If $X$ is furthermore Fano, then $P_k$ are all zero.
\item The above expression $\PO(L(\bu),b)$ only depends on $X$(and independent of $\bu$), so we denote the above $\PO(L(\bu),b)$ just by $\PO(x_1,\cdots,x_n)$, or $\PO$ for short.
\end{itemize}
\end{theorem}

We are particularly interested in $Crit(\PO)$ because of the following fact:
\begin{theorem}\cite{CO,FOOOtoric1}
If $b=\sum a_i \be_i \in H^1(L(\bu),\Lambda_0)$ then the following are equivalent:
\begin{itemize}
\item $\displaystyle \frac{\partial\PO}{\partial x_i}(b)=0$ for all $i=1,\cdots,n$.
\item $HF((L(\bu),b),(L(\bu),b);\Lambda) \cong H^*(T^n;\Lambda).$
\item $HF((L(\bu),b),(L(\bu),b)) \neq 0.$
\end{itemize}
\end{theorem}
%We call such torus fibers {\em strongly balanced}.

Finally we recall the following generation result.
\begin{theorem}[\cite{AFOOO,EL}]
For a compact toric Fano manifold $X$, we have a decomposition
\[D^\pi Fu(X)=\bigoplus\limits_{\eta \in Crit(\PO)}D^\pi Fu_\eta(X)\]
and every torus fiber with a weak bounding cochain in $Crit(\PO)$ generates $D^\pi Fu_\eta(X)$.
\end{theorem}

\subsection{Closed-open maps}\label{subsec:comaps}
Let $\CM_{k+1,l}^{main}(\beta)$ be the moduli space of holomorphic discs of class $\beta$ with $k+1$ boundary marked points respecting the cyclic order and $l$ interior marked points. Define the evaluation map 
\[ ev=(ev_1^+,\cdots,ev_l^+,ev_0,\cdots,ev_k): \CM^{main}_{k+1,l}(\beta) \to X^l \times L^{k+1}\]
and consider the map \cite{FOOO,FOOOtoric2}:
\[ \fq_{l,k;\beta}: E_l(\mathcal{A}) \otimes H^*(L;\Lambda)^{\otimes k} \to H^*(L;\Lambda),\]
\begin{align*}
 &\fq_{l,k;\beta}([\alpha_1 \otimes \cdots \otimes \alpha_l],p_1,\cdots,p_k) \\
 :=&(ev_0)_*(ev_1^+\times \cdots \times ev_l^+ \times ev_0 \otimes \cdots \otimes ev_k)^*([\alpha_1 \OCO \alpha_l],p_1 \OCO p_l).\end{align*}
Here $E_l (\mathcal{A})$ means the subspace of $\mathcal{A}^{\otimes l}$, consisting of $S_l$-invariant elements and 
\[ [\alpha_1 \OCO \alpha_l]=\sum_{\sigma\in S_k}\frac{1}{l!}\alpha_{\sigma(1)} \OCO \alpha_{\sigma(l)}.\]
We similarly define 
\[ \fq_{l,k;\beta}([\alpha_1 \OCO \alpha_l], p_1, \cdots, p_k) \] for $p_1,\cdots,p_k$ being transverse intersections among distinct Lagrangians $L_0,\cdots,L_k$, using moduli spaces of holomorphic polygons with faces on $L_0,\cdots,L_k$ together with interior marked points. We also define (both for a single Lagrangian or a transversal collection of Lagrangians)
\[ \fq_{l,k}([\alpha_1\OCO \alpha_l],p_1,\cdots,p_k):=\sum_{\beta \in \pi_2(M,\vec{L})} \fq_{l,k;\beta}([\alpha_1\OCO \alpha_l],p_1,\cdots,p_k)\cdot T^{\omega(\beta)}.\]

\begin{theorem}[\cite{FOOOtoric2}]
Let $X$ be a compact toric manifold. Then for any $\mathfrak{b} \in \mathcal{A}(\Lambda_+)$ and $b \in H^1(L(\bu);\Lambda)$ the following always holds:
\[ \sum_{l=0}^\infty \sum_{k=0}^\infty \fq_{l,k}(\mathfrak{b}^l;b^k) \equiv 0 \; \mod \; \Lambda_+ \be.\]
We denote the left hand side by $\PO^\fb(b) \cdot \be$ and call the coefficient $\PO^\fb(b)$ a {\em bulk-deformed potential.}
\end{theorem}
Now we define the following.
\begin{definition}
Let $X$ be a compact toric manifold and $\fb \in \CA(\Lambda_+)$. Let $b_0,\cdots,b_k$ be weak bounding cochains. Then the following operations define a new $\AI$-structure:
\[m_k^{\fb,\vec{b}}(x_1,\cdots,x_k):= \sum_{l=0}^\infty \fq_{l,*}(\fb^l;e^{b_0},x_1,e^{b_1},x_2,\cdots,e^{b_{k-1}},x_k,e^{b_k}).\]
\end{definition}
Let $\fb=t\alpha$ where $\alpha \in \mathcal{A}$ and $t$ is a formal variable, and $b=\sum x_i \be_i \in H^1(L(\bu),\Lambda_0)$ for some $\bu \in Int(P)$. Then $\PO^{\fb}(b)$ is a formal power series in $t$. We observe the following:
\[ \widetilde{\ks}(\alpha):= \frac{\partial \PO^{\fb}(b)}{\partial t}\big|_{t=0}=\sum_{k\geq 0}\fq_{1,k}(\alpha;b^k).\]
By above theorem, $\widetilde{\ks}(\alpha)$ is a multiple of the unit $\be$. Again, the coefficient is a power series in $x_1,\cdots,x_n$, or a Laurent polynomial in $y_1,\cdots,y_n$. We will later employ $\widetilde{\ks}$ for the statement of the closed string mirror symmetry for toric Fano manifolds.

We can also think of the $t$-derivative of $\AI$-equations of $\{m_k^{t\alpha,\vec{b}}\}$ at $t=0$ as follows(we omit the decoration $e^{b_i}$ for simplicity).
\begin{align}
 0=&\sum (-1)^{|x_1|'+\cdots+|x_i|'}m_{i+k-j+1}\big(x_1,\cdots,x_i,\fq(\alpha;x_{i+1},\cdots,x_j),x_{j+1},\cdots,x_k\big)\label{qinfty}\\
 & + \sum (-1)^{|x_1|'+\cdots+|x_i|'} \fq(\alpha;x_1,\cdots,x_i,m_{j-i}(x_{i+1},\cdots,x_j),x_{j+1},\cdots,x_k). \nonumber
\end{align}
This equation will be crucially used in Section \ref{sec:mainthm}. Remark that it is related to the degeneration of holomorphic discs with one interior constraint(see Figure \ref{fig:intdegeneration}).

\begin{figure}
\includegraphics[height=2.5in]{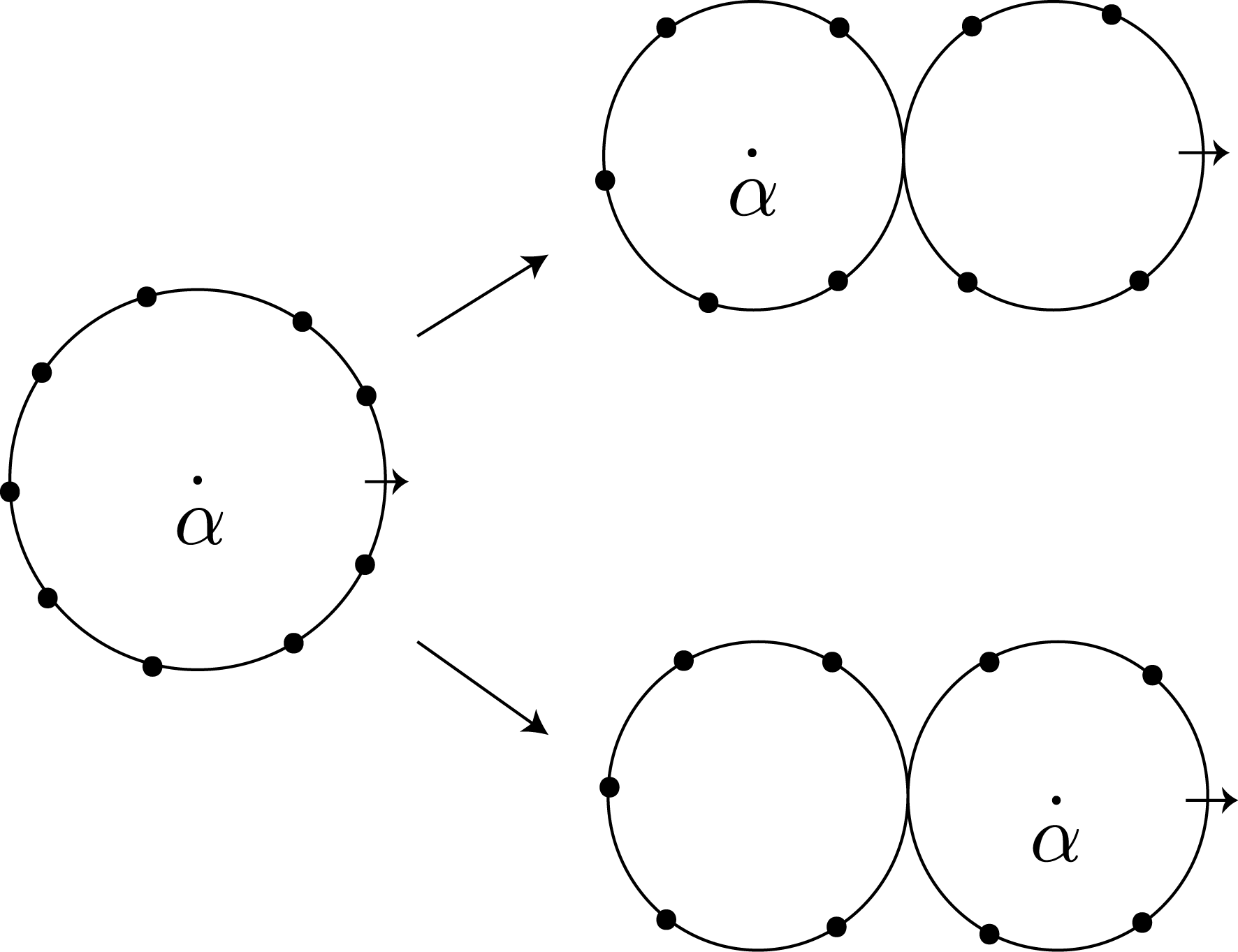}
\caption{A degeneration of holomorphic discs with interior constraint $\alpha$. In broken discs, components without interior constraints correspond to $m_k(\vec{p})$ and components with interior constraints correspond to $\fq_{1,k}(\alpha;\vec{p})$.}
\label{fig:intdegeneration}
\end{figure}

\section{Category of matrix factorizations}\label{sec:mfctg}
Let $R$ be a commutative algebra and $W\in R$ be a non-zero-divisor. A {\em matrix factorization} of $W$ is a $\ZZ/2$-graded $R$-module $E = E^0 \oplus E^1$
 with a degree 1 endomorphism 
 \[ Q= \left( {\begin{array}{cc} 0 & Q_{01} \\ Q_{10} & 0 \ \end{array} } \right), \; \textrm{where} \; Q_{ij} \in \Hom(E^j,E^i), \]
 satisfying $Q^2 = W \cdot Id$. We denote the above data by $(E,Q)$ for short.%For simplicity, we assume that 
%$W$ has only one critical value at $0$ in this section.   {\bf (Novikov coefficients?)}
%\begin{definition}\label{def:mf}
%Fix $\lambda \in \CC$.
% A {\em matrix factorization} of $W-\lambda$ is a $\ZZ/2$-graded trivial vector bundle $E = E^0 \oplus E^1$
% with an endomorphism 
% \[ Q= \left[ {\begin{array}{cc} 0 & Q_{01} \\ Q_{10} & 0 \ \end{array} } \right], \; \textrm{where} \; Q_{ij} \in \Hom(E^j,E^i), \]
% satisfying $Q^2 = (W-\lambda) \cdot Id$.  In other words, $Q_{01}Q_{10} = Q_{10}Q_{01} = (W-\lambda) \cdot Id$.
% \end{definition}

Matrix factorizations of $W$ form a differential $\ZZ/2$-graded category $MF(R,W)$ as follows: given two matrix factorizations $(E,Q), (F,Q')$, $\ZZ/2$-graded morphisms from $(E,Q)$ to $(F,Q')$ are
given by homomorphisms 
 \[ \Phi= \left( {\begin{array}{cc} \Phi_{00} & \Phi_{01} \\ \Phi_{10} & \Phi_{11} \ \end{array} } \right), \; \textrm{where} \; \Phi_{ij} \in \Hom(E^j,F^i). \]
Compositions of morphisms are defined in the obvious way.
The differential on a morphism is defined as
\[ \delta \Phi := Q \Phi  - (-1)^{|\Phi|} \Phi Q\]
for morphisms of homogeneous degrees.

 We often write as $MF(W)$ instead of $MF(R,W)$. The cohomology category $[MF(W)]:=H^0(MF(W))$ with respect to $\delta$ is a triangulated category.
The following theorem illustrates the geometric relevance of categories of matrix factorizations.
\begin{theorem}[\cite{Or}]
$[{MF}(W)] \simeq D_{sg}(Spec(R/W)).$
\end{theorem}

$D_{sg}(X)$ is defined for a variety $X$ by the quotient of $D^b(X)$ by $Perf(X)$ which is a thick subcategory consisting of complexes which are quasi-isomorphic to bounded complexes of locally free sheaves(in affine case, projective modules). If $X$ is smooth then $D_{sg}(X)$ is trivial, hence the category roughly measures how the variety is singular. 

We are particularly interested in the case $W \in R=\Lambda[x_1^\pm,\cdots,x_n^\pm]$ with isolated critical points. According to the above theorem of Orlov, $[MF(W-\lambda)]$ is nontrivial if and only if $\lambda$ is a critical value of $W$. For a critical point $\eta$ of $W$, let 
\[ \hat{R_\eta}:=\Lambda[[x_1-\eta_1,\cdots,x_n-\eta_n]].\] 
If there are finitely many critical points, we take a direct sum of categories over all critical points as 
\[\overline{MF}(W):=\bigoplus\limits_{\eta \in Crit(W)} MF(\hat{R_\eta},W-W(\eta)).\] The direct sum of categories means that there are no nontrivial morphisms between objects in different summands. %We remark that even for some distinct critical points $\eta$ and $\eta'$ with the same critical value $\lambda$, a matrix factorization $P$ of $W-\lambda$ become different when it is localized at $\eta$ and $\eta'$.

\section{Mirror symmetry}\label{sec:mirror}

\subsection{Closed string mirror symmetry}
Let $X$ be compact toric Fano. One of the simplest version of closed string mirror symmetry is the following:
\begin{conjecture}
There is a ring homomorphism
\[QH^*(X) \cong Jac(\PO)\]
where $W$ is the Landau-Ginzburg mirror.
\end{conjecture}
There were some related results due to \cite{Ba,Gi}. While these approaches are explicit and given by algebraic methods, the following Fukaya-Oh-Ohta-Ono's construction of the isomorphism is rather geometric. %In particular, Batyrev computed quantum cohomologies of toric Fano manifolds giving generators and relations which can be read off in terms of data of moment polytopes. Then he gave the isomorphism algebraically, comparing generators and relations on both sides.

\begin{theorem}[\cite{FOOOtoric3}]
For a compact toric manifold $X$(not necessarily Fano), the ring isomorphism above can be given by
\begin{align*} 
 \ks: QH^*(X)& \to Jac(\PO), \\
\alpha &\mapsto [\widetilde{\ks}(\alpha)]
\end{align*}
where $\PO$ is the leading order potential and
\[Jac(\PO)=\frac{\Lambda[y_1^\pm,\cdots,y_n^\pm]}{\Big(y_1\frac{\partial \PO}{\partial y_1},\cdots,y_n\frac{\partial \PO}{\partial y_n}\Big)}.\]
\end{theorem}
Recall the definition of $\widetilde{\ks}$ from Section \ref{subsec:comaps}: we counted holomorphic discs with interior constraints.

\subsection{Localized mirror functor}
%We recall the construction of localized mirror functor from \cite{CHL1},\cite{CHL2}.
%Roughly speaking, given a weakly unobstructed object in an $\AI$-category, the functor from
%this category to the matrix factorization category of the associated potential function is obtained by
%a version of obstructed Yoneda embedding. To set the notation and our convention, we recall the basic setup, but
%to match with the algebraic construction of Shklyarov, we will use a different convention from the above references.

%\begin{definition}
%A filtered $\AI$-category $\mathcal{C}$ consists of a collection of objects, and for each pair of objects $C_1,C_2$,
%morphism space $\mathcal{C}(C_1,C_2)$ is given by a graded torsion free filtered $\Lambda_0$-module, with
%a sequence of degree one operations
%$$ m_k : \mathcal{C}[1](C_0,C_1) \otimes \mathcal{C}[1](C_{k-1},C_k) \to \mathcal{C}[1](C_1,C_k)$$
%for $k=0,1,\ldots$. Here $m_k$ respects the filtration and satisfies the $\AI$-equations. Namely
%for $x_i \in \mathcal{C}[1](C_{i-1},C_{i})$, $i=1,\cdots,k$, 
%%$\epsilon_1 = \sum_{j=1}^{i-1}(|x_j|')$
%\begin{equation*}
%\sum_{k_1+k_2=k+1} \sum_{i=1}^{k_1} (-1)^{\sum_{j=1}^{i-1}(|x_j|')} m_{k_1}(x_1,\cdots,x_{i-1},m_{k_2}(x_i,\cdots,x_{i+k_2-1}),x_{i+k_2},\cdots,x_k)=0.
%\end{equation*}
%\end{definition}

Given a potential function $W \in R$, we modify the convention of the dg category $(MF(W),\delta,\circ)$ to obtain an $\AI$-category $(\mfai(W),m_1^{MF},m_2^{MF})$. Objects are still matrix factorizations of $W$, and morphism spaces are
\[\Hom_{\mfai(W)}((E,Q_E), (F,Q_F)): =  Hom_R(F,E).\]
$m_1^{MF}$ and $m_2^{MF}$ are defined as 
\[m_1^{MF} (\Phi) := \delta (\Phi) = Q_E \circ \Phi - (-1)^{|\Phi|}\Phi \circ Q_{F},\]
\[m_2^{MF}(\Phi,\Psi) := (-1)^{|\Phi|} \Phi \circ \Psi.\]
Then $m_1^{MF}$ and $m_2^{MF}$ satisfy $\AI$-relations (\ref{aieq}). We also define
\[ \overline{\mfai}(W):=\bigoplus\limits_{\eta \in Crit(W)} \mfai(\hat{R_\eta},W-W(\eta)).\]

Let $(\mathbb{L},b)$ be a weakly unobstructed Lagrangian where $b=x_1 X_1 + \cdots + x_n X_n$, and let $W=m_0^b$ be a function on the space of weak Maurer-Cartan solutions. For any other weakly unobstructed Lagrangian $L$ with potential value $\lambda$
(i.e. there exists a Maurer-Cartan element $b_0$ such that $m_0^{b_0} = \lambda$),
the $\AI$-equation gives the following matrix facatorization identity
\[(m_1^{b_0,b})^2 = W - \lambda.\]

\begin{theorem}[\cite{CHL1}]
We have an $\AI$-functor $ \locmir^\LL$ from Fukaya $\AI$-category to $\AI(MF(W))$.
On the level of objects, it sends a weakly unobstructed Lagrangian $L$ with Maurer-Cartan element $b_0$ to
the matrix factorization $(E,Q)$, which is given by 
$$E:=\Hom\big((L,b_0), (\mathbb{L},b)\big),\; Q:=-m_1^{b_0,b}$$
On the level of morphisms, 
 $\locmir^\LL_k$ is defined as
 \begin{equation}\label{lmfunctor}
\locmir^\LL_k (x_1,\cdots, x_k) (\bullet) := m_{k+1}(x_1,\cdots,x_k,\bullet).
 \end{equation}

 Then these data define an $\AI$-functor.
\end{theorem}
\begin{remark}
In \cite{CHL1,CHL2}, they considered $CF\big((\mathbb{L},b), (L,b_0)\big)$ instead. In the current convention,
we do not have any sign in \eqref{lmfunctor}.
\end{remark}
\begin{proof}
To avoid confusion, let $\{m_k^{Fu}\}$ be $\AI$-operations of the Fukaya category, while $m_1^{MF}$ and $m_2^{MF}$ are $\AI$-operations of $\mfai(W)$. We need to show that
\begin{align}
& \locmir^\LL( \hat{m}^{Fu}(x_1,\cdots,x_k)) \nonumber \\ 
=& \sum_{\vec{x}_1\otimes \vec{x}_2=x_1\otimes\cdots\otimes x_k} m_2^{MF}( \locmir^\LL(\vec{x}_1), \locmir^\LL(\vec{x}_{2}))  
+ m_1^{MF} (\locmir^\LL(x_1,\cdots,x_k)). \label{lmai} 
\end{align}
The left hand side can be written as
\[\sum_{\vec{x}_1\otimes\vec{x}_2\otimes\vec{x}_3=x_1\otimes\cdots\otimes x_k}(-1)^{|\vec{x}_1|'} m^{Fu}\big(\vec{x}_1 \otimes m^{Fu}(\vec{x}_{2}) \otimes \vec{x}_3,\bullet\big).\]
By definition of $m_2^{MF}$,
\begin{align*} 
&\sum_{\vec{x}_1\otimes \vec{x}_2=x_1\otimes\cdots\otimes x_k} m_2^{MF}( \locmir^\LL(\vec{x}_1), \locmir^\LL(\vec{x}_{2}))\\
=&\sum_{\vec{x}_1\otimes \vec{x}_2=x_1\otimes\cdots\otimes x_k}(-1)^{|\vec{x}_1|' +1 } \locmir^\LL(\vec{x}_1) \circ  \locmir^\LL(\vec{x}_2) \\
=&\sum_{\vec{x}_1\otimes \vec{x}_2=x_1\otimes\cdots\otimes x_k} 
(-1)^{|\vec{x}|' +1 } m^{Fu}(  \vec{x}_1, m^{Fu}( \vec{x}_2, \bullet)).
\end{align*}
The last term equals to
\[- m_1^{Fu} m_{k+1}^{Fu}(x_1,\cdots,x_k,\bullet) - (-1)^{|x_1|'+\cdots+|x_k|'+1} m_{k+1}^{Fu} (x_1,\cdots,x_k,-m_1^{Fu}(\bullet)).\]
%due to the fact that the map $m_{k+1}(x_1,\cdots,x_k,\bullet)$ has degree $|x_1|'+\cdots+|x_k|'+1$.

Summarizing, the equation (\ref{lmai}) is nothing but an $\AI$-relation of the Fukaya category, and we are done.
\end{proof}

For compact toric Fano manifolds, the above construction indeed gives homological mirror symmetry.

\begin{theorem}[\cite{CHL2}]
For any compact toric Fano manifold $X$, the localized mirror functor 
\[\locmir^\LL: Fu(X) \to \overline{\mfai}(\PO)\]
is an $\AI$-equivalence for the monotone fiber $\LL$ with $b=x_1 e_1 + \cdots+ x_n e_n$, where $\{e_i\}$ is a basis of $H^1(\LL,\Lambda)$ and $x_i$ are formal variables.
\end{theorem}

\section{Main result}\label{sec:mainthm}
Recall the closed-open map $\CO: QH^*(X) \to HH^*(Fu(X)).$ We give a mirror counterpart $\gamma: Jac(\PO) \to HH^*(\overline{\mfai}(\PO))$ of the map $\mathcal{CO}$ by the following lemma. Recall that the notation $\OL{\mfai}(\PO)$ means a direct sum of categories
\[\OL{\mfai}(\PO)=\bigoplus\limits_{\eta\in Crit(\PO)}\mfai\big(\hat{R_\eta},\PO-\PO(\eta)\big).\]
 Also recall that if $\PO$ has isolated critical points we have a decomposition of $Jac(\PO)$ into local rings
\[ Jac(\PO)=\bigoplus\limits_{\eta \in Crit(\PO)} Jac(\PO;\eta).\]
For simplicity let $\CA:=Fu(X)$, $\CB:=\OL{\mfai}(\PO)$ and $\CB_\eta:=\mfai\big(\hat{R_\eta},\PO-\PO(\eta)\big)$.

\begin{lemma}
The following map is a well-defined ring isomorphism:
\[ \gamma: Jac(\PO) \to \bigoplus\limits_{\eta\in Crit(\PO)}HH^*(\CB_\eta),\;  
[r] \mapsto \sum_{\eta\in Crit(\PO)}[\phi_{r_\eta}]\]
where 
\[ [r]=\sum_{\eta\in Crit(\PO)} [r_\eta] \] for $[r_\eta] \in Jac(\PO;\eta)$, and
\[ \phi_{r_\eta}=\bigoplus\limits_{X \in Ob(\CB_\eta)}r_\eta\cdot \id_X \in CH^*(\CB_\eta)\]
which is a Hochschild cocycle with length zero part only.
\end{lemma}

\begin{proof}
Let $\mathcal{C}=\mfai(\hat{R},\PO)$, where $\hat{R}$ is the completion of $R$ at the ideal $\mathfrak{m}$ which supports a critical point of $\PO$.
Recall from \cite{Dyc} that the isomorphism $\gamma:Jac(\PO) \to HH^*(\calC)$ is induced by the map
\[ Jac(\PO) \to \RR Hom_{\calC-\calC}(\calC_\Delta,\calC_\Delta),\; [r] \mapsto (\calC_\Delta \stackrel{r\cdot}{\longrightarrow} \calC_\Delta).\]
Replacing with the bar resolution $B(\calC_\Delta)$, the map $r\cdot: \calC_\Delta \to \calC_\Delta$ lifts to the map of complexes of $\calC$-bimodules
\[\xymatrix{
\cdots \ar[r] & \bigoplus\limits_{B_0,B_1 \in \calC} \calC(-,B_0) \otimes \calC(B_0,B_1) \otimes \calC(B_1,-) \ar[r] \ar[d]& \bigoplus\limits_{B \in \calC} \calC(-,B) \otimes \calC(B,-) \ar[r] \ar[d]^{r\cdot m_2^\calC} & 0 \\
0 \ar[r] & 0\ar[r] & \calC_\Delta\ar[r] & 0 }\]
and by the following identification
\begin{align*}
 hom_{\calC-\calC}\big(\calC(-,B_0) \otimes \calC(B_0,B_1) \OCO \calC(B_{p-1},B_p) \otimes \calC(B_p,-),X\big) \\ \simeq hom_k \big(\calC(B_0,B_1) \otimes \cdots \otimes \calC(B_{p-1},B_p),X(B_0,B_p)\big)
 \end{align*}
the above map of complexes changes to
\[\xymatrix{
\cdots \ar[r] & \bigoplus\limits_{B_0,B_1 \in \calC} \calC(B_0,B_1) \ar[r] \ar[d] & \bigoplus\limits_{B \in \calC}k \ar[r] \ar[d]^-{r} & 0
\\
0 \ar[r] & 0\ar[r] & \bigoplus_{B \in \calC} \calC(B,B) \ar[r] & 0 } \]
where the map $r$ sends $1$ to $r\cdot \id_B$ for each $B$.
\end{proof}

We justify that the map $\gamma$ is indeed the closed-open map on the $B$-model as follows: given the open-closed map $\mathcal{OC}: HH_*(Fu(X)) \to QH^*(X)$ on the $A$-model, let $\sigma \in HH_*(Fu(X))$ be the preimage of the unit $1 \in QH^*(X)$. For $\psi=\UL{a_0}\otimes a_1\otimes \cdots \otimes a_n \in HH_*(\calC)$ with $\calC$ being an $\AI$-category, recall the {\em cap product}
\[ - \cap \psi: HH^*(\calC) \to HH_*(\calC),\]
\[ \phi \cap \psi:= \sum (-1)^\star m_*^\calC\big(a_{l+1},\cdots,a_i,\phi(a_{i+1},\cdots,a_j)\otimes a_{j+1}\otimes \cdots \otimes a_n\otimes \UL{a_0}\otimes\cdots \otimes a_k\big)\otimes a_{k+1}\otimes \cdots \otimes a_l\]
where
\begin{align*}
 \star=&\; |\phi|'(|a_0|'+|a_1|'+\cdots+|a_i|')\\
 &+(|a_{l+1}|'+\cdots+|\phi(a_{i+1},\cdots,a_j)|'+\cdots+|a_n|')(|a_0|'+|a_1|'+\cdots+|a_l|').\end{align*}
Via the cap product, $HH_*(\calC)$ is equipped with a module structure over $HH^*(\calC)$. Then we have the following important fact:
\[ \mathcal{OC} \circ (\cap \sigma) \circ \mathcal{CO}=\id.\] 
In other words, $\mathcal{CO}=(\cap\sigma)^{-1}\circ \mathcal{OC}^{-1}$ as a linear map. 

Now consider the ``open-closed map" on the $B$-model
\[ \xi_\eta: HH_*(\CB_\eta) \to Jac(\PO;\eta)\] which is explained in \cite{CLS}(for original descriptions see \cite{Se,Shk}). Since it is an isomorphism, let $\psi:=\xi_\eta^{-1}(1) \in HH_*(\CB_\eta)$. Let $\psi=\UL{a_0}\otimes a_1 \otimes \cdots \otimes a_n.$ Pick $r_\eta \in Jac(\PO;\eta)$. Since $\gamma(r_\eta)=\bigoplus_{X\in Ob(\CB_\eta)}r_\eta\cdot \id_X$, we have the following computation of cap products:
\[ \gamma(r_\eta) \cap \psi= m_2(r_\eta \cdot \id,\UL{a_0})\otimes a_1 \otimes \cdots \otimes a_n=r_\eta \cdot \UL{a_0} \otimes a_1 \otimes\cdots\otimes a_n,\]
so the following is straightforward:
\begin{equation}\label{hochmodule} \xi_\eta \circ (\cap \psi) \circ \gamma=\id.\end{equation}
We can summarize the meaning of the above arguments as follows: since $HH_*(\overline{\mfai}(\PO))$ is a module over $HH^*(\overline{\mfai}(\PO))$ which is isomorphic to $Jac(\PO)$ via $\gamma$, $HH_*(\overline{\mfai}(\PO))$ is also equipped with a $Jac(\PO)$-module structure. On the other hand, there is also a canonical $Jac(\PO)$-module structure on $HH_*(\overline{\mfai}(\PO))$ by multiplication. %The module action is well-defined because any partial derivative of $\PO$ defines a nullhomotopic morphism. 
The above condition (\ref{hochmodule}) tells us that these two $Jac(\PO)$-module structures coincide.

%Summarizing, the main feature of the map $\gamma$ is that it is a map which defines a $Jac(W)$-module structure on $HH_*(\overline{\mfai}(W))$ via $HH^*(\calC)$-module structure, which is identical to the canonical $Jac(W)$-module structure on $HH_*$.

\begin{theorem}
Let $X$ be a compact toric Fano manifold and $\PO$ be its LG superpotential. Then the following diagram commutes:
\[\xymatrix{
QH^*(X) \ar[r]^-{\mathcal{CO}} \ar[dd]^{\ks} &HH^*(Fu(X))\ar[rd]^\cong_{[L_\CM^1]} & \\
& & Hom_{\mathcal{A}-\mathcal{B}-bimod}(\CM,\CM) \\
Jac(\PO) \ar[r]^-{\gamma} & HH^*(\OL{MF_{\AI}}(\PO)) \ar[ur]^\cong_{[R_\CM^1]} & }
\]
so $\ks$ becomes a ring isomorphism.
\end{theorem}

\begin{proof}
%If $W$ is the LG superpotential of a compact toric Fano manifold, we know that every critical point of $W$ is Morse, i.e. $Jac(W;\eta)$ has dimension 1 for any $\eta \in Crit(W)$. It follows that any element in $Jac(W;\eta)$ is represented by a scalar $c \in \Lambda$. If we have different scalars $c \neq c'$, then $\phi_c \neq \phi_{c'}$, so $\gamma_\eta:Jac(W;\eta) \to HH^*(\CB_\eta)$ is an injective linear map. Since we already know that they are isomorphic, $\gamma_\eta$ must be an isomorphism because both sides are finite dimensional. Finally, define 
%\[\gamma:=\bigoplus_{\eta \in Crit(W)} \gamma_\eta.\]
Let $\locmir^\LL: Fu(X) \to \OL{\mfai}(\PO)$ be the localized mirror functor whose reference $\LL$ is the monotone torus fiber. Recall our notations $\CA=Fu(X)$ and $\CB=\OL{\mfai}(\PO)$. For a compact toric Fano manifold $X$, $\locmir^\LL$ is actually an $\AI$-quasiequivalence by the result of \cite{CHL2}, and the Morita bimodules are given by
\[ \CM=(\locmir^\LL\otimes 1)^*\CB_\Delta,\;
 \CN=(1 \otimes \locmir^\LL)^*\CB_\Delta\]
so that $\CM \otimes_\CB \CN \simeq \CA_\Delta$ and $\CN \otimes_\CA \CM \simeq \CB_\Delta.$

Let $\CF_\alpha:=(L_\CM^1 \circ \mathcal{CO})(\alpha) \in hom^\bullet_{\CA-\CB}(\CM,\CM)$ for $\alpha \in QH^*(X).$ Let $(a_i : L_i \to L_{i+1})_{i=1,...,r}$ be a tuple of morphisms in $Fu_\lambda(X)$, fix a reference Lagrangian $\LL$ which has a critical value $\lambda$ and let
\[\UL{m}\in \CM(L_{r+1},P)=\CB_\Delta(\locmir^\LL(L_{r+1}),P)=\CB(\locmir^\LL(L_{r+1}),P)\] 
for some $P \in Ob(\CB).$
\begin{align}
& \CF_\alpha^{r|1|0}(a_1,\cdots,a_r,\UL{m}) \nonumber \\
=& L_\CM^1(\CO(\alpha))(a_1,\cdots,a_r,\UL{m}) \nonumber \\
=& \sum (-1)^{|a_1|'+\cdots+|a_i|'}\mu_\CM^{i+r-j+1|1|0}(a_1,\cdots,a_i,\fq(\alpha;a_{i+1},\cdots,a_j),a_{j+1},\cdots,a_r,\UL{m}) \nonumber \\
=& \sum (-1)^{|a_1|'+\cdots+|a_i|'}m_2^{\CB}\Big(\locmir^\LL\big(a_1,\cdots,a_i,\fq(\alpha;a_{i+1},\cdots,a_j),a_{j+1},\cdots,a_r\big),\UL{m}\Big) \label{lm1}
\end{align}
and $\CF_\alpha^{r|1|s}=0$ if $r$ and $s$ are both nonzero, since $\CB$ has $\AI$-operations only up to $m_2$.

Also, for $\CG_\alpha:=(R_\CM^1 \circ \gamma \circ \ks)(\alpha) \in hom^\bullet_{\CA-\CB}(\CM,\CM)$,
\begin{align*}
 \CG_\alpha^{0|1|0}(\UL{m}) = R^1_\CM(\gamma(\ks(\alpha)))(\UL{m})= (-1)^{|\UL{m}|}m_2^{\CB}\big(\UL{m},\ks(\alpha)\cdot \id_P\big)
\end{align*}
and $\CG_\alpha^{r|1|s}=0$ if $r\neq 0$ or $s\neq 0$. The sign $(-1)^{|\UL{m}|}$ is due to the definition of $R^p_\CM$ in Lemma \ref{lmrm}.

We show that \[\CG_\alpha-\CF_\alpha=\delta \xi_\alpha\] for some $\xi_\alpha \in hom^\bullet_{\CA-\CB}(\CM,\CM)$, so $[\CF_\alpha]=[\CG_\alpha]$ in $Hom_{\CA-\CB}(\CM,\CM).$
For any $r\geq 0$, let 
\[\xi_\alpha^{r|1|0}(a_1,\cdots,a_r,\UL{m}):=m_2^{\CB}(\fq(\alpha;a_1,\cdots,a_r,\bullet),\UL{m})\] 
and $\xi_\alpha^{r|1|s}=0$ if $s\neq 0$. Then $|\xi_\alpha|=0$. Here, $\fq(\alpha;a_1,\cdots,a_r,\bullet)$ is a morphism in $\CB$ from $CF(L_1,\LL)$ to $CF(L_{r+1},\LL)$, i.e. the bullet means an input in $CF(L_{r+1},\LL)$.

First let $r\geq 1$. Then $(\CF_\alpha-\CG_\alpha)^{r|1|0}=\CF_\alpha^{r|1|0}$, and continuing from (\ref{lm1}),
\begin{align}
&\CF_\alpha^{r|1|0}(a_1,\cdots, a_r,\UL{m}) \nonumber \\
=& \sum (-1)^{|a_1|'+\cdots+|a_i|'}m_2^\CB(\locmir^\LL(a_1,\cdots,a_i,\fq(\alpha;a_{i+1},\cdots,a_j),a_{j+1},\cdots,a_r),\UL{m}) \nonumber\\
=& \sum (-1)^{|a_1|'+\cdots+|a_i|'} m_2^\CB(m_{i+r-j+1}^\CA(a_1,\cdots,a_i,\fq(\alpha;a_{i+1},\cdots,a_j),\cdots,a_r,\bullet),\UL{m}) \nonumber \\
=& \sum m_2^\CB\Big((-1)^{|\vec{a}_1|'+1}\fq(\alpha;\vec{a}_1,m^\CA(\vec{a}_2,\bullet))
	+(-1)^{|\vec{a}_1'|'+1}\fq(\alpha;\vec{a}_{1}',m^\CA(\vec{a}_{2}'),\vec{a}_{3}',\bullet),\UL{m}\Big) \label{moduleandxi} \\
&+ \sum (-1)^{|\vec{a}_1|'+1}m_2^\CB\Big(m^\CA(\vec{a}_1,\fq(\alpha;\vec{a}_2,\bullet)),\UL{m}\Big)\label{xiandmodule} \\
&\pm \sum m_2^\CB\Big(\fq(\alpha;a_1,\cdots,a_r,\bullet,m_0^\LL(1)),\UL{m}\Big) \label{qunit} \\
&\pm \sum m_2^\CB\Big(m^\CA_{r+2}(a_1,\cdots,a_r,\bullet,\fq_0^\LL(\alpha)),\UL{m}\Big).\label{aiunit}
\end{align}
Recall that the identity at (\ref{moduleandxi}) is given by the formula (\ref{qinfty}). Also observe that 
\[(\ref{moduleandxi})=-(-1)^{|\xi_\alpha|}(\xi_\alpha \circ \hat{\mu}_\CM)(a_1,\cdots,a_r,\UL{m}).\]
On the other hand, the following also holds
\[(\ref{xiandmodule})=-(\mu_\CM \circ \hat{\xi_\alpha})(a_1,\cdots,a_r,\UL{m}),\]
by computations below:
\begin{align*}
 & \sum m_2^\CB\Big((-1)^{|\vec{a}_1|'+1}m^\CA(\vec{a}_1,\fq(\alpha;\vec{a}_2,\bullet)),\UL{m}\Big) \\
 =& \sum m_2^\CB \Big((-1)^{|\vec{a}_1|'+1}m^\CA(\vec{a}_1,\bullet) \circ \fq(\alpha;\vec{a}_2,\bullet),\UL{m}\Big) \\
 =& \sum m_2^\CB \Big( m_2^\CB\big(m^\CA(\vec{a}_1,\bullet),\fq(\alpha;\vec{a}_2,\bullet)\big),\UL{m}\Big) \\
 =& \sum (-1)^{|\vec{a}_1|'}m_2^\CB \Big(m^\CA(\vec{a}_1,\bullet),m_2^\CB\big(\fq(\alpha;\vec{a}_2,\bullet),\UL{m}\big)\Big) \\
 =& \sum (-1)^{|\vec{a}_1|'+|\vec{a}_1|'+1}m_2^\CB(m^\CA(\vec{a}_1,m_2^\CB\big(\fq(\alpha;\vec{a}_2,\bullet),\UL{m}\big)\Big) \\
 =& -(\mu_\CM\circ \hat{\xi_\alpha})(a_1,\cdots,a_r,\UL{m}).
\end{align*}

Furthermore, (\ref{qunit}) and (\ref{aiunit}) vanish since $m_0^\LL(1)$ and $\fq_0^\LL(\alpha)$ are both constant multiples of $\AI$-unit. Hence,
\[ (\CG_\alpha-\CF_\alpha)^{r|1|0}=(\delta \xi_\alpha)^{r|1|0} \;{\rm for}\; r>0.\]

If $\UL{m} \in \CB(\locmir^\LL(L),P)$, then we have
\begin{align}
& (\CF_\alpha-\CG_\alpha)^{0|1|0}(\UL{m}) \nonumber\\
=& m_2^\CB(\LM^\LL(\fq_0^L(\alpha)),\UL{m})-(-1)^{|\UL{m}|}m_2^\CB(\UL{m},\ks^\lambda(\alpha)\cdot \id_P) \nonumber \\
=& m_2^\CB(\LM^\LL(\fq_0^L(\alpha)),\UL{m})-\UL{m}\circ (\ks^\lambda(\alpha)\cdot \id_P )\nonumber \\
=& m_2^\CB\Big(m_2^\CA\big(\fq_0^L(\alpha),\bullet\big)-(-1)^{|\bullet|}m_2^\CA\big(\bullet,\fq_0^\LL(\alpha)\big),\UL{m}\Big) \label{fminusgm}
\end{align}
and by (\ref{qinfty}) again, 
\begin{align}
&(\ref{fminusgm}) \nonumber \\
=& -m_2^\CB\Big(\fq(\alpha;m_1^\CA(\bullet))+m_1^\CA(\fq(\alpha;\bullet)),\UL{m}\Big) \label{qm1m1q} \\ 
&-m_2^\CB\Big(\fq(\alpha;m_0^L(1),\bullet)+(-1)^{|\bullet|'}\fq(\alpha;\bullet,m_0^\LL(1)),\UL{m}\Big)\label{qunit2}
\end{align} 
and since $m_0^L(1)$ and $m_0^\LL(1)$ are both (multiples of) $\AI$-units, (\ref{qunit2}) vanish. On the other hand, 
\begin{align}
& (\delta\xi_\alpha)^{0|1|0}(\UL{m}) \nonumber \\
=& \xi_\alpha(m_1^\CB(\UL{m}))+ m_1^\CB(\xi_\alpha(\UL{m})) \nonumber \\
=& m_2^\CB\Big(\fq(\alpha;\bullet),m_1^\CB(\UL{m})\Big) + m_1^\CB\Big(m_2^\CB(\fq(\alpha;\bullet),\UL{m})\Big) \label{deltaxi}
\end{align}
and we observe that
\[ \fq(\alpha;-m_1^\CA(\bullet))-m_1^\CA(\fq(\alpha;\bullet))=m_1^\CB(\bullet \mapsto \fq(\alpha;\bullet)),\]
%When we wrote $m_1^\CA(\fq(\alpha;\bullet))$, it is a Floer coboundary output in $CF(L,\LL)$, while $m_1^\CB(\fq(\alpha;\bullet))$ means the differential of the morphism $\fq(\alpha;\bullet):CF(L,\LL) \to CF(L,\LL)$ of matrix factorizations. 
but by $\AI$-relation on $\CB$ we have
\[ m_2^\CB\Big(m_1^\CB(\fq(\alpha;\bullet)),\UL{m}\Big)+m_2^\CB\Big(\fq(\alpha;\bullet),m_1^\CB(\UL{m})\Big)+m_1^\CB\Big(m_2^\CB\big(\fq(\alpha;\bullet),\UL{m}\big)\Big)=0,\]
hence
\[ (\ref{qm1m1q})=- (\ref{deltaxi})=-(\delta\xi_\alpha)^{0|1|0}(\UL{m}).\]

Finally, we show $(\delta\xi_\alpha)^{r|1|s}=0$ for $s\neq 0$, so that in this case
\[(\CG_\alpha-\CF_\alpha)^{r|1|s}=(\delta\xi_\alpha)^{r|1|s}\]
where the left hand side is automatically zero by definition of $\CG_\alpha$ and $\CF_\alpha$.

Since $\CB$ has no $\AI$-operations $m_{\geq 3}$, we only need to compute $(\delta\xi_\alpha)^{r|1|1}$.
\begin{align}
& \delta\xi_\alpha(a_1,\cdots,a_r,\UL{m},b) \nonumber\\
=& (-1)^{|a_1|'+\cdots+|a_r|'}\xi_\alpha(a_1,\cdots,a_r,m_2^\CB(\UL{m},b)) \nonumber\\
&+\xi_\alpha(\hat{m}(a_1,\cdots,a_r),\UL{m},b) \label{r11}\\
& +\sum (-1)^{|a_1|'+\cdots+|a_i|'}\xi_\alpha(a_1,\cdots,a_i,\mu_\CM(a_{i+1},\cdots,a_r,\UL{m}),b)\label{r11'}\\ 
&+ m_2^\CB(\xi_\alpha(a_1,\cdots,a_r,\UL{m}),b). \nonumber
\end{align}
By property $\xi_\alpha^{r|1|s}=0$ for $s\neq 0$, (\ref{r11}) and (\ref{r11'}) are zero, and
\[ (-1)^{|a_1|'+\cdots+|a_r|'}
\xi_\alpha(a_1,\cdots,a_r,m_2^\CB(\UL{m},b))
=(-1)^{|a_1|'+\cdots+|a_r|'}m_2^\CB\Big(\fq(\alpha;a_1,\cdots,a_r,\bullet),m_2^\CB(\UL{m},b)\Big),\]
\[ m_2^\CB\Big(\xi_\alpha(a_1,\cdots,a_r,\UL{m}),b\Big)=m_2^\CB\Big(m_2^\CB\big(\fq(\alpha;a_1,\cdots,a_r,\bullet),\UL{m}\big),b\Big).\]
The sum of above two terms is zero due to the $\AI$-relation of $m_2^\CB$, hence 
\[\delta\xi_\alpha(a_1,\cdots,a_r,\UL{m},b)=0.\]

Summarizing all above arguments, we conclude that 
\[ (\CG_\alpha-\CF_\alpha)^{r|1|s}=(\delta\xi_\alpha)^{r|1|s}\]
for any $r$ and $s$, so on the cohomology level,
\[ [\CF_\alpha]=[\CG_\alpha].\]
\end{proof}

\bibliographystyle{amsalpha}

\end{document}